\documentclass[11pt,reqno,a4paper]{amsart}

\usepackage{amssymb,epsfig,graphics}
\usepackage{amsmath, float}

\newtheorem{theorem}{Theorem}[section]
\newtheorem{proposition}[theorem]{Proposition}

\newtheorem{lemma}[theorem]{Lemma}
\newtheorem{sub-lemma}[theorem]{Sub-Lemma}

\newtheorem{remark}[theorem]{Remark}

\DeclareMathOperator{\esssup}{ess sup}
\DeclareMathOperator{\osc}{osc}
\DeclareMathOperator{\Lip}{Lip}

\let\eps=\varepsilon

\def\1{{{\mathit 1} \!\!\>\!\! I} }

\newcommand{\tpsi}{\tilde\psi}

\begin{document}
\title{On the statistical stability of Lorenz attractors with a $C^{1+\alpha}$ stable foliation}
\author{Wael Bahsoun}
\author{Marks Ruziboev}
\address{Department of Mathematical Sciences, Loughborough University,
Loughborough, Leicestershire, LE11 3TU, UK}
\email{W.Bahsoun@lboro.ac.uk}
\email{M.Ruziboev@lboro.ac.uk}
\thanks{WB and MR would like to thank The Leverhulme Trust for supporting their research through the research grant RPG-2015-346.}
\keywords{Lorenz flow, Statistical stability.}
\subjclass{Primary 37A05, 37C10, 37E05}
\begin{abstract}
We prove statistical stability for a family of Lorenz attractors with a $C^{1+\alpha}$ stable foliation. \end{abstract}
\date{\today}
\maketitle
\markboth{Wael Bahsoun \and Marks Ruziboev}{ Lorenz attractors with a $C^{1+\alpha}$ stable foliation}
\bibliographystyle{plain}
\tableofcontents
\section{Introduction}
In his seminal work \cite{Lo} Lorenz introduced the following system of equations 
\begin{equation}\label{LS}
\begin{cases}
\dot x = -10x+10y \\
\dot y = 28x-y-xz \\
\dot z = -\frac{8}{3}z+xy
\end{cases}
\end{equation}
as a simplified model for atmospheric convection. Numerical analysis performed by Lorenz showed that the above system exhibits sensitive dependence on initial conditions and has a non-periodic ``strange" attractor. A rigorous mathematical framework of similar flows was initiated with the introduction of the so called  geometric Lorenz flow in \cite{ABS, GW}. Nowadays it is well known that the geometric Lorenz attractor, whose vector field will be denoted by $X_0$, is robust in the $C^1$ topology \cite{AP}. This means that vector fields $X_\eps$ that are sufficiently close in the $C^1$ topology to $X_0$ admit invariant contracting foliations $\mathcal F_{\eps}$ on  the Poincar\'e section $\Sigma$ and they admit strange attractors. More precisely, there exists an open neighbourhood $U$ in $\mathbb R^3$ containing $\Lambda$, the attractor of $X_0$, and an open neighbourhood $\mathcal U$ of $X_0$ in the $C^1$ topology such that for all vector fields $X_\eps\in \mathcal U$, the maximal invariant set $\Lambda_{X_\eps}=\cap_{t\ge 0}X_\eps^t(U)$ is a transitive set which is invariant under the flow of $X_\eps$ \cite{MPP}. The papers \cite{Tu1, Tu} provided a computer-assisted proof that the classical Lorenz flow; i.e., the flow defined in \eqref{LS} has a robustly transitive invariant set containing an equilibrium point. In \cite{LMP} it was proved that the Lorenz flow is mixing. Statistical limit laws were first obtained in \cite{HM}. Then rapid mixing for the Lorenz attractor and statistical limit laws for their time-1 maps was obtained in \cite{AMV}. Recently, Ara\'ujo and Melbourne proved in \cite{AM1} that the stable foliation of the Lorenz flow is $C^{1+\alpha}$. Moreover, their methods also imply that $C^1$ perturbations $ X_{\eps}$ of the Lorenz flow admit a $C^{1+\alpha}$ stable foliation and the stable foliation of $X_\eps$ is $C^1$ in $\eps$ (see Theorem 2.2 in \cite{Bort}). Further, in another paper Ara\'ujo and Melbourne \cite{AM2} showed that the Lorenz system is exponentially mixing and that this property is robust in the $C^1$ topology.

\bigskip

In this paper we study a family of perturbations $X_\eps$ which are consistent with the results of \cite{AM1, AM2,  AP, Tu} and prove statistical stability of Lorenz attractors with a $C^{1+\alpha}$ stable foliation. For a precise statement see Theorem \ref{main} in section \ref{setup}. Previous results on the statistical stability of Lorenz attractors was announced in \cite{AS} but only for flows with $C^2$ stable foliations. In \cite{GL} among other things, statistical stability of Poincar\'e maps for `$BV$-like' Lorenz attractors is studied. In our work, we only assume $C^{1+\alpha}$ stable foliation for the family of flows and we prove that the corresponding $1$-d maps are strongly statistically stable. We then obtain statistical stability for the family of flows. Our proofs allow the discontinuities in the base of the corresponding  Poincar\'e maps to change with the perturbation (See Figure \ref{1dfig} for an illustration). In fact, this is the main issue with perturbations of Lorenz systems since the derivative blows up at the discontinuity point. 
\bigskip

The rest of the paper is organised as follows. In section \ref{setup} we introduce the family of flows that we study in this paper. The statement of our main result (Theorem \ref{main}) is also in this section.  In section \ref{pf} we provide proofs of our result in a series of lemmas and propositions.  

\section{Setup and statement of main result}\label{setup}
\subsection{A geometric Lorenz flow with a  $C^{1+\alpha}$ stable foliation}
Let $X_0:\mathbb R^3\to\mathbb R^3$ be a vector field associated with a flow that has an equilibrium point at $0$. We assume that $X_0$ satisfies the following assumptions:\\

\noindent$\bullet$ The differential $DX_0(0)$ has three real eigenvalues $\lambda_2<\lambda_3<0<\lambda_1$, $\lambda_1+\lambda_3>0$ (Lorenz-like singularity) and $\lambda_1+\lambda_2<\lambda_3$\footnote{Note that this condition is required to get the $C^{1+\alpha}$ regularity of the stable foliation for the flow  (see \cite{AM1} section 5).}.\\
$\bullet$ Let
$$
\Sigma :=\left\{(x, y, 1)|-\frac 1 2\le x, y \le\frac 1 2 \right\}
$$
and let $\Gamma=\{(0, y, 1)| -\frac 1 2 \le y \le \frac 1 2\}$ be the intersection of $\Sigma$ with the local stable manifold of the fixed point $0$. 
The segment $\Gamma$ divides $\Sigma$ into two parts
$$
\Sigma^+=\{(x, y, 1)\in \Sigma:\, x>0\}  \,\, \text{and} \,\,   \Sigma^-=\{(x, y, 1)\in \Sigma:\, x<0\}.  
$$
There exists a well defined Poincar\'e map $F:\Sigma\to\Sigma$ such that the images of $\Sigma^{\pm}$ by this map are curvilinear triangles $S^\pm$ as in Figure \ref{Poincare},  without the vertexes $(\pm 1, 0, 1)$ and every line segment in $\mathcal F=\{(x, y, 1) \in \Sigma \mid x=\text{const}\}$ except $\Gamma$ is mapped into a segment $\{(x, y, 1) \in S^\pm \mid  x=\text{const}\}$. The return time $\tau : \Sigma\setminus \Gamma\to \mathbb R$  to $S^\pm$ is given by $\tau(x, y, 1)=-\frac 1 {\lambda_1} \log|x|$. 

\noindent$\bullet$ The flow maps $\Sigma$ into $S^\pm$ in a smooth way so that the Poincar\'e map $F(x, y)= X_0^{\tau(x, y, 1)}((x, y, 1), t)$ has the form 
$$
F(x, y)=(T(x), g(x, y)),
$$

 \begin{itemize}
 	\item $T:I\to I$, where $I:=[-\frac{1}{2},\frac12]$, has discontinuity at $x=0$ with side limits $T(0^+)=-\frac 1 2$ and $T(0^-)=\frac 1 2$;
 	\item $T$ is piecewise monotone increasing and $T$ is $C^1$ on $I\setminus\{0\}$;
	\item $\lim_{x\to 0^\pm}T'(x)=+\infty$;
	\item $\frac{1}{T'}$ is $\alpha$-H\"older  on $I_i$, $i=1,2;$ $I_1:=[-\frac12, 0]$ and $I_2:=[0, \frac12]$ with $0<\alpha\le 1$;  
	\item  There are $C>0$ and $\theta>1$ such that $(T^n)'(x)\ge C\theta^{n}$ for any $n\in \mathbb N$;
 	\item $T$ is transitive; 
 	\item $g$  preserves $\mathcal F$ and it is uniformly contracting; i.e., there exists $K>0$ and $0<\rho<1$ such that for any given leaf $\gamma$ of the foliation and $\xi_1, \xi_2\in\gamma$ and $n\ge 1,$ 
 	$$
 	\text{dist}(F^n(\xi_1), F^n(\xi_2))\le K\rho^n\text{dist}(\xi_1, \xi_2).
 	$$
 	 \end{itemize}
	
	\begin{figure}[H] 
   \centering
   \includegraphics [width=3in]{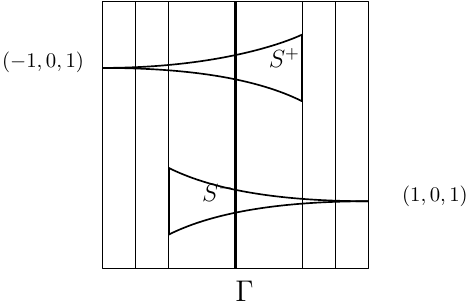} 
   \caption{The images $S^\pm$ of   $\Sigma^\pm$ under the Poincar\'e map.}   \label{Poincare}
\end{figure} 
	 \subsection{Universally bounded $p$-variation} We now define a space that captures the regularity of $\frac{1}{T'}$. For $p\ge 1$, we say $f:I \to \mathbb R$ is a function of universally bounded $p$-variation if 
	$$
	V_p(f):=
	\sup_{-1/2\le x_0<...<x_n\le 1/2}
	\left(
	\sum_{i=1}^{n}\left| f(x_i)-f(x_{i-1})\right|^p
			\right)^{1/p} <+\infty.
	$$ 
	
	The space of universally bounded $p$-variation functions is denoted by $UBV_p(I)$ and it will play a key role in studying perturbations of $X_0$.
\subsection{Perturbations: a family $\mathcal X$ of flows with $C^{1+\alpha}$ stable foliations}\label{familyofflows}
 We now consider perturbations of $X_0$ which are consistent with the results of \cite{AM1, AP}. From now on we are going to set $p:=\frac{1}{\alpha}$. Let $\mathcal X$ be the family of $C^1$ perturbations of $X_0$; i.e., there exists an open neighbourhood $U$ in $\mathbb R^3$ containing $\Lambda$, the attractor of $X_0$, and an open neighbourhood $\mathcal U$ of $X_0$ containing $\mathcal X$ such that 
	\begin{enumerate}
	\item for each $X_\eps\in \mathcal X,$ the maximal forward invariant set $\Lambda_{X_\eps}$ is contained in $U$ and is an attractor containing a hyperbolic singularity; 
	\item for each $X_\eps\in\mathcal X,$ $\Sigma$ is a cross-section for the flow with a return time $\tau_\eps$ and a Poincar\'e map $F_\eps;$
	\item for each $X_\eps\in\mathcal{X},$ the map $F_\eps$ admits a $C^{1+\alpha}$ uniformly contracting invariant foliation $\mathcal F_\eps$  on $\Sigma$;
	\item $F_\eps:\Sigma\to\Sigma$ is given by\footnote{By c) $g_\eps$ is $C^{1+\alpha}$. Moreover, it is a uniform contraction on stable leaves.}
	$$F_\eps(x,y)=(T_\eps(x), g_{\eps}(x,y));$$
	\item \label{expansion} the map $T_\eps:I\to I$ is transitive piecewise $C^{1}$ expanding with two branches and a discontinuity point $O_\eps$ such that $T(O_{\eps}^+)=-\frac12$, $T(O_{\eps}^-)=\frac12$ and $\lim_{x\to O_\eps^\pm} T'_\eps(x)=+\infty$ (see Figure \ref{1dfig} for an illustration);  
\item	for any $\eta>0$ there exists $\eps_0$ and an interval $H(\eps_0)\subset I$ such that for all $0\le \eps<\eps_0$, $\{O_\eps, 0\}\in H,  \,\, |H|=2|O_{\eps_0}| <\eta;$ and 
$$d(T, T_\eps):=\displaystyle\sup_{x\in H^c} \{|T_\eps(x)-T(x)|+|T'_\eps(x)-T'(x)| \} \le C\eta,$$
where $H^c:=I\setminus H$;
	\item there are uniform (in $\eps$ and $x$) constants $C>0$ and $\theta>1$ such that $(T_\eps^n)'(x)\ge C\theta^{n}$ except at the discontinuity point $O_\eps$;
	\item \label{VI} there is a uniform (in $\eps$) constant $W>0$ such that 
	$${\underset{i\in{1,2}}{\max}}V_{p|_{ I_{i,\eps}}}(\frac{1}{T'_{\eps}})\le W,$$
	where $I_{i,\eps}$ is a monotonicity interval of $T_\eps$, and $V_{p}(\cdot)$ is the $p$-variation;
	\item  for  any $n>0$ let $\mathcal P_\eps^{(n)}:=\vee_{j=0}^{\ell-1}T^{-j}_\eps(\mathcal P_\eps)$, where $\mathcal P_\eps=\{I_{1,\eps}, I_{2,\eps}\}$. There exists $\delta_n>0$ independent of $\eps$ such that $\min_{J\in \mathcal P_\eps^{(n)}}|J|\ge \delta_n$;
	\item the return time $\tau_{\eps}:\Sigma\setminus\Gamma_\eps\to\mathbb R$ satisfies the following: there is a constant $C>0$ such that for each $X_\eps\in\mathcal X$
	$$
	\tau_\eps(\xi)\le -C\log|\pi_\eps(\xi)-O_\eps|,
	$$
where $\pi_\eps$ is the projection along the leaves of $\mathcal F_\eps$ onto $I$.
	\end{enumerate}
\begin{figure}[htbp] 
   \centering
   \includegraphics[width=2.5in]{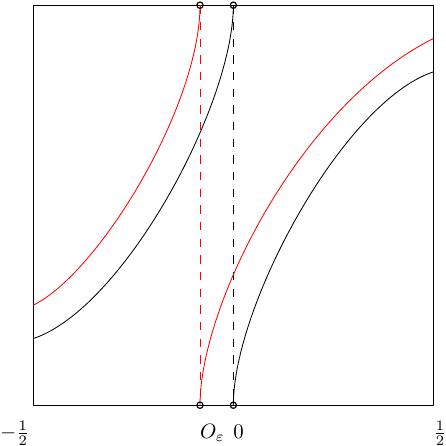} 
   \caption{A graph of $T_{\eps}$ with discontinuity $O_\eps$ versus the graph of $T$ with discontinuity at $0$.}   \label{1dfig}
\end{figure}
\begin{remark}\label{comparison}
In \cite{AS} the authors impose more regularity conditions on the stable foliations and consequently on $T_\eps$. In particular, they assume that $T_\eps$ is piecewise $C^2$. They also relay on the result of \cite{Kel82} which assumes that $T_0$ and $T_\eps$ are close in the Skorohod distance and that the transfer operators admit a uniform, in $\eps$, Lasota-Yorke inequality. See \cite{Kel82} $\S$ 3); in particular the cautionary Remark 15, items (ii) and (iii). In our work, we only assume that the map $T_\eps$ is piecewise $C^1$ (see condition e)) and ${\underset{i\in{1,2}}{\max}}V_{p|_{ I_{i,\eps}}}(\frac{1}{T'_{\eps}})\le W,$ for some $W>0$ (see condition h)). We also assume that the maps are close in sense of assumption f). We would also like to stress that in our setting $\sup_{x\in [-1,O_\eps)}|T'_{\eps}(x)|=\sup_{x\in (O_\eps,1]}|T'_{\eps}(x)|=\infty$ and hence $T'_\eps$ does not admit a uniform H\"older constant on its domain. 
\end{remark}	
Before stating our main result, we define an appropriate Banach space, which was first introduced by Keller \cite{K}, that will play a key role in our analysis.
\subsection{A Banach space} Let $S_\rho(x):=\{y\in I: |x-y|<\rho\}$ and $f:I\to \mathbb R$ be any function defined on $I.$ Let 
$$
 \osc(f, \rho, x):= \esssup\{|f(y_1)-f(y_2)|: y_1, y_2\in S_\rho(x)\}, 
 $$ 
and
$$ \osc_1(f, \rho)=\|\osc(f, \rho, x)\|_1,$$
where the essential supremum is taken with respect to the two dimensional Lebesgue measure on $I\times I$ and $\|\cdot\|_1$ is the $L^1$- norm  with respect to Lebesgue measure on $I.$ Fix $\rho_0>0$ and let $BV_{1, 1/p}\subset L^1$ be the Banach space equipped with the norm
	$$
	\|f\|_{1, 1/p}=V_{1, 1/p}(f)+\|f\|_1,
	$$
	where 
	$$
	 V_{1, 1/p}(f)=\sup_{0<\rho\le \rho_0}\frac{\osc_1(f, \rho)}{\rho^{1/p}}.
	$$
Notice that $V_{1, 1/p}(\cdot)$ depends on $\rho_0$. The fact that  $BV_{1, 1/p}$ is a Banach space is proved in \cite{K}. Moreover, it is proved in \cite{K} that the unit ball of $BV_{1, 1/p}$ is compact in $L^1$. We now list several inequalities, involving functions in  $BV_{1, 1/p}$, that were proved in \cite{K}. For any $p\ge 1$ and $f\in UBV_p(I)$ we have 
\begin{equation}\label{fact2}
 V_{1, 1/p}(f)\le 2^{1/p}V_p(f).
\end{equation}
Moreover,  if $f\in BV_{1, 1/p}$ and  $Y\subset I$ is an interval with $|Y|\ge 4\rho_0$, then  for each $0<\rho\le \rho_0$ we have 
	\begin{equation}\label{fact1}
	\osc_1(f\cdot \mathbf 1_{Y}, \rho)\le 
	\left(2+\frac{8\rho_0}{|Y|-2\rho_0}\right)\int_Y\osc(f|_{Y}, \rho, x)dx+\frac{4\rho}{|Y|}\int_Y|f(x)|dx.
	\end{equation}
Further, if $Y, Z\subset I$ are intervals such that $T:Y\to Z$ is differentiable then 
\begin{align}\label{fact3}
\begin{split}
&\int_Z\osc(\frac{f}{T'}\circ T^{-1}|_{Z}, \rho, z)dx\le \int_Y\osc(f|_Y, (C\theta)^{-1}\rho, y)dy \\
&\hskip 0.6cm+ 5\int_Z\osc(\frac{1}{T'}\circ T^{-1}|_{Z}, \rho, z)dx\left(\frac{1}{|Y|}\int_Y|f|dx+\frac{1}{\rho_0}\int_Y\osc(f|_Y, \rho_0, y)dy\right).
\end{split}
\end{align}

\subsection{Statement of the main results}
We first recall the definition of statistical stability for continuous-time dynamical systems: Let $V$ be a neighbourhood of $0$. Let $(X_\eps)_{\eps\in V}$ be a family of flows which is endowed with some topology $\mathfrak T$.  Assume that every $X_\eps$ admits a unique SRB measure\footnote{For more information about SRB measures, we refer to \cite{Y}.} $\mu_\eps$. The family $(X_\eps)_{\eps\in V}$ is called statistically stable if $\eps\mapsto X_\eps$ is continuous at $\eps=0$ in the weak $\ast$-topology, i.e.  
$$\lim_{\eps\to 0}\int f d\mu_\eps=\int fd\mu_0,$$
for any continuous function $f:\mathbb R^3\to\mathbb R$. Statistical stability is defined analogously for discrete-time dynamical systems. We refer the reader to the articles \cite{A, AV} for more information. 
\begin{theorem}\label{main}
Let $X_\eps \in \mathcal X$. Then
\item[1)] $X_{\eps}$ admits a unique invariant probability SRB measure $\mu_\eps$.
\item[2)] For any continuous $\varphi:\mathbb R^3\to\mathbb R$ we have
$$\lim_{\eps\to 0}\int\varphi d\mu_{\eps}=\int\varphi d\mu;$$
where $\mu$ is the SRB measure associated with $X_0$; i.e. the family $\mathcal X$ is statistically stable.
\end{theorem}
\begin{remark}
1) in Theorem \ref{main} is well known. See for instance \cite{APPV}. We prove 2) in the following section.
\end{remark}
\section{Proofs}\label{pf}
\subsection{Statistical stability of the family $\mathcal X$}
Let $\bar\mu_\eps$ and $\bar\mu_0$ denote the unique absolutely continuous invariant measures of the one dimensional maps $T_\eps,T_0$ respectively. Let $h_\eps, h_0$ denote the densities corresponding to $\bar\mu_\eps, \bar\mu_0$ respectively. For $\eps\ge 0$, let
$$P_{\eps}:L^1(I)\to L^1(I)$$ 
denote the transfer operator(Perron-Frobenius) associated with $T_\eps$ \cite{Ba, BG}; i.e., for any $f\in L^1(I)$
$$P_\eps f(x)=\sum_{i=1}^{2}(f\cdot\frac{1}{T'_\eps})(T_\eps|I_{i,\eps})^{-1}(x)\cdot \mathbf{1}_{T_\eps(I_{i,\eps})}(x).$$ 
Our first goal is to prove  that $\lim_{\eps\to 0}||h_\eps-h||_1=0.$ This will be achieved by showing that $P_\eps$, when acting on $BV_{1, 1/p}$, satisfies a uniform (in $\eps)$ Lasota-Yorke inequality and that $\eps\mapsto P_\eps$ is continuous at $\eps=0$ in an appropriate topology. We will be then in a setting where we can apply the spectral stability result of \cite{KL}, and hence achieve our first goal. 
\subsubsection{A uniform Lasota-Yorke inequality} In this subsection, we show that $P_\eps$ admits a uniform (in $\eps$) Lasota-Yorke inequality when acting on $BV_{1,1/p}$. We first start with two lemmas to control, uniformly in $\eps$, the $p$-variation of $\frac{1}{(T^n_\eps)'}$. 
\begin{lemma}\label{Tell}
For any $\ell\in\mathbb N$ 
$$\max_{J\in\mathcal P^{(\ell+1)}_\eps}V_{p| J}\left(\frac{1}{|(T_\eps^{\ell+1})'|}\right)<\frac{W(\ell)}{C\theta^{\ell}},$$
 where $W(\ell):= \frac{C^{\ell}+C-2}{C^{\ell-1}(C-1)}W$ and $W$ is as in assumption $h)$.
\end{lemma}
	\begin{proof}
	The proof is by induction on $\ell$. 
	Conclusion holds for $\ell=0$  by h).  Suppose it holds for $\ell$. 
	Since  $J\in \mathcal P_{\eps}^{(\ell+1)}$ is a subset of some  $I_{i, \eps} \in\mathcal  P_\eps$  and  $T_{\eps}(J)\in \mathcal P_\eps^\ell$, using the standard properties of variation 
	we get:
	\begin{equation*}\begin{split}
	V_p\left(\frac{1}{|(T^{\ell+1}_\eps)'|}\right)
	&\le 
	\sup_{x\in I} \frac{1}{|T_\eps'|}V_p\left(\frac{1}{|(T_\eps^\ell)'\circ T_\eps|}\right)+
	\sup_{x\in I}\left(\frac{1}{|(T_\eps^\ell)'\circ T_\eps|}\right)V_p\left(\frac{1}{|T_\eps'|}\right)\\
	&\le 
	\frac{1}{C\theta}V_p\left(\frac{1}{|(T_\eps^\ell)'\circ T_\eps|}\right)+ \frac{1}{C\theta^\ell}W\\
	&\le	\frac{1}{C\theta}\cdot\frac{1}{C\theta^{\ell-1}}\cdot\frac{C^{\ell-1}+C-2}{C^{\ell-2}(C-1)}W+\frac{1}{C\theta^{\ell}}W\\
	&\le\frac{1}{C\theta^\ell}\frac{C^\ell+C-2}{C^{\ell-1}(C-1)}W.
	\end{split}\end{equation*}
	\end{proof}	
\begin{lemma}\label{pvariation}
	Fix $\ell$ such that  $\bar\theta:=C\theta^{\ell-1}>1$. 
	Then for any $\eps$ and  $n$   the following holds
	$$
	\max_{J\in\mathcal P^{(n\ell)}} V_{p | J}\left(\frac{1}{(T^{n\ell}_\eps)'}\right)\le 
	\frac{n}{\bar\theta^{n}}W(\ell-1),
	$$
	where $W(\ell)$ is as in Lemma \ref{Tell}.
	\end{lemma}
	
	\begin{proof}
	The proof is again  by induction on $n.$ For $n=1$, by Lemma \ref{Tell}, we have 
	$$
	\max_{J\in\mathcal P^{(\ell)}_\eps}V_{p | J}\left(\frac{1}{|(T_\eps^{\ell})'|}\right)\le W(\ell-1)/(C\theta^{\ell-1})=W(\ell-1)/\bar{\theta}.
	$$
Suppose that
$$
\max_{J\in \mathcal P^{((n-1)\ell)}_\eps}V_{p | J}\left(\frac{1}{|(T_\eps^{(n-1)\ell})'|}\right)\le
	\frac{n-1}{\bar\theta^{n-1}}W(\ell-1).
	$$		
Let $J$ be any interval in $\mathcal P^{(n\ell)}_\eps$ then $T_\eps^\ell J\in\mathcal P^{(n-1)\ell}_\eps$. Hence, we get: 
\begin{equation*}\begin{aligned}
	V_{p, J}\left(\frac{1}{|(T_\eps^{n\ell})'|}\right)
	&\le 
	\sup_{x\in J} \frac{1}{|(T^{\ell}_\eps)'|}V_{p, J}\left(\frac{1}{|(T_\eps^{(n-1)\ell)})'\circ T_\eps^\ell|}\right) \\
	& +\sup_{x\in J}\left(\frac{1}{|(T_\eps^{(n-1)\ell})'\circ T_\eps|}\right)V_{p, J}\left(\frac{1}{|(T^\ell_\eps)'|}\right)\\
	&\le \frac{1}{C\theta^\ell}\frac{n-1}{\bar\theta^{n-1}}W(\ell-1)+
	\frac{1}{(C\theta^\ell)^{n-1}}\frac{1}{\bar\theta}W(\ell-1)\\
		&\le\frac{n-1}{\bar\theta^{n}}W(\ell-1)+\frac{1}{\bar\theta^{n}}W(\ell-1) =\frac{n}{\bar\theta^{n}}W(\ell-1).
	\end{aligned}\end{equation*}	
\end{proof}
\begin{lemma}\label{uniform} 
There exists\footnote{All the constants in this lemma are independent of $\eps$.} $\rho_0$, $0<A_1, A_2<\infty$, $0<\kappa<1$, such that for any $0<\rho<\rho_0$, $n\in\mathbb N$ and for any  $f\in BV_{1, 1/p}$ the following holds 
\begin{equation*}
||P^n_\eps f||_{1,1/p}\le A_1\kappa^n||f||_{1, 1/p}+A_2||f||_1.
\end{equation*}
\end{lemma}
\begin{proof}
We first obtain an inequality for $n=1$. Let  
\begin{align*}
\underline K:=\inf_{0\le\eps< \eps_0}\{|T_{\eps}I_{1, \eps}|, |T_{\eps}I_{2, \eps}|\}.
\end{align*}
Fix $\rho_0$ such that
\begin{equation}\label{ch.ofepsilon}
\rho_0< \underline K/{10}.
\end{equation}

Since $\esssup(f+g)\le \esssup(f)+\esssup(g)$ we have 
\begin{equation}\label{oscestimate}
\begin{aligned}
&\osc_1(P_\eps f, \rho) \le \sum_{i=1}^{2}\osc_1\left(f \cdot \frac{1}{T_\eps'}\circ(T_\eps|I_{i, \eps})^{-1}\cdot \mathbf{1}_{T_\eps (I_{\eps, i}),}\,\,  \rho \right)\\
&\le \sum_{i=1}^{2}
\left(
    \left(2+\frac{8\rho_0}{|T(I_{i, \eps})|-2\rho_0}
    \right)  \right.
   \int_{T_\eps(I_{i, \eps})}\osc(\frac{f}{T_\eps'}\circ(T_\eps|I_{i, \eps})^{-1}|_{T_\eps(I_{i, \eps})}, \rho, x)dx   
 \\
& \left.
   \hskip 5.5cm+\frac{4\rho}{|T_\eps(I_{i, \eps})|}\int_{T_\eps(I_{i, \eps})}|f\circ (T|I_{i, \eps})^{-1}|dx\right).
       \end{aligned}
       \end{equation}
The last inequality follows from inequality \eqref{fact1}. Now, using the notation $T_\eps|I_{i, \eps}=T_{i, \eps}$  and  change of variable formula we have 
\begin{equation}\label{changeofvar}
\int_{T(I_{i, \eps})}\left(|f|\circ T_{i, \eps}^{-1}\right)
\left(\frac{1}{T_{i, \eps}'}\circ T_{i, \eps}^{-1}\right)dx=\int_{I_{i, \eps}}|f|dx.
\end{equation}
Inequality \eqref{ch.ofepsilon} implies that 
\begin{equation}\label{deltave0}
2+\frac{8\rho_0}{|T_\eps(I_{i, \eps})|-2\rho_0} <  3.
\end{equation} 
On the other hand from \eqref{fact3} it follows that  
\begin{equation}\label{oscsemifinal}
\begin{aligned}
\int_{T_\eps (I_{i, \eps})}\osc\left(\frac{f}{T_\eps'}\circ T_{i, \eps}^{-1}|_{T_\eps(I_{i, \eps})}, \rho, x\right)dx \le 
\int_{I_{i, \eps}}\osc(f|_{I_{i, \eps}},  (C\theta)^{-1}\rho, x)dx  \\
+5\int_{T_\eps (I_{i, \eps})}\osc \left(\frac{1}{T_{i, \eps}'}\circ T^{-1}_{i, \eps}|_{T_\eps (I_{i, \eps})}, \rho, x\right) dx \\
\cdot \left(\frac{1}{|I_{i, \eps}|}\int_{I_{i, \eps}}|f|dx+\frac{1}{\rho_0}\int_{I_{i, \eps}}\osc(f|_{I_{i, \eps}}, \rho_0, y)dy
\right).
\end{aligned}\end{equation}
Using the relation between $L^p$ norms\footnote{ $\| f \|_1 \le  \mu(X)^{1-1/p}\|f\|_p$, where $X$ is the space and $\mu$ is a measure on it.}, the definition of $V_{1, 1/p}(\cdot)$ and \eqref{fact2} lead to
\begin{equation*}
\begin{aligned}
&\int_{T_\eps (I_{i, \eps})}\osc \left(\frac{1}{T_{i, \eps}'}\circ T^{-1}_{i, \eps}|_{T_\eps (I_{i, \eps})}, \rho, x\right) dx\\
&\le \left(\int_{T_\eps (I_{i, \eps})}\osc \left(\frac{1}{T_{i, \eps}'}\circ T^{-1}_{i, \eps}|_{T_\eps I_{i, \eps}}, \rho, x\right) dx \right)^{1/p} |T_\eps (I_{i, \eps})|^{1-1/p} \\
&\le \rho^{1/p}V_{1, 1/p}\left(\frac{1}{T_{i, \eps}'}\circ T^{-1}_{i, \eps}|_{T(I_{i, \eps})}\right)\le (2\rho)^{1/p}V_p\left(\frac{1}{T_{i, \eps}'}\circ T^{-1}_{i, \eps}|_{T(I_{i, \eps})}\right)\\
& \le (2\rho)^{1/p}V_p\left((T_{i, \eps}')^{-1}|_{I_{i, \eps}}\right).
\end{aligned}\end{equation*}
Therefore,
\begin{equation}\label{oscfinal}
\osc_1\left(\frac{1}{T_{i, \eps}'}\circ T_{i, \eps}^{-1}|_{T_\eps (I_{i,\eps})}, \rho\right)\le 2\rho^{1/p}W.
\end{equation}
Substituting equation \eqref{oscfinal} first into \eqref{oscsemifinal}  and then substituting  \eqref{changeofvar},  \eqref{deltave0} and \eqref{oscsemifinal} into  \eqref{oscestimate} and using property h) gives 
\begin{align*}
\osc_1(P_\eps f, \rho)& \le  3\sum_{i=1}^{2}\int_{I_i}\osc(f|_{I_{i, \eps}}, (C\theta)^{-1}\rho, y)dy\\
&+30\rho^{1/p}W\sum_{i=1}^2\left(\frac{1}{|I_{i, \eps}|}\int_{I_{i, \eps}}|f|dm+\frac{1}{\rho_0}\int_{I_{i, \eps}}\osc(f|_{I_{i, \eps}}, \rho_0, y)dy
\right)\\
&+\sum_{i=1}^{2} \frac{4\rho}{|T_\eps(I_{i, \eps})|}\int_{I_{i, \eps}}|f|dx\\
&\le 3\int_I \osc(f,  \frac{\rho}{C\theta}, x)dx \\
&+30W\rho^{1/p}
\left(
\frac{1}{\delta_1}\int_{I}|f|dx+\frac{1}{\rho_0}\int_I\osc(f, \rho_0, x) dx
\right)+ \frac{4\rho}{ \underline K }\int_I|f|dx.
\end{align*}
Therefore,  
\begin{equation}\begin{aligned}
\frac{\osc_1(P_\eps f, \rho)}{\rho^{1/p}}\le \left(\frac{3}{C\theta}
+30 W\rho^{1/p-1}\right)V_{1, 1/p}(f)\\
+\left(30 W\frac{1}{\delta_1}+\frac{4\rho^{1-1/p}}{ \underline K }\right)\|f\|_1.
\end{aligned}\end{equation}
Consequently, we have
\begin{equation}\label{step1}
\begin{split}
V_{1, 1/p}(P_\eps f)
&\le \left(\frac{3}{C\theta}
+30 W\rho^{1/p-1}\right)V_{1, 1/p}(f)\\
&\hskip 4cm+\left(30 W\frac{1}{\delta_1}+\frac{4\rho^{1-1/p}}{ \underline K }\right)\|f\|_1.
\end{split}
\end{equation}
We now prove an inequality for all $n$ as stated in the lemma. Fix $\ell\in\mathbb N$ such that $\bar\theta:=C\theta^{\ell-1}>1$. By Lemma \ref{pvariation} and \eqref{step1} applied to $(P_\eps^\ell)^k$ we get
\begin{equation}\label{step2}
\begin{aligned}
V_{1, 1/p}((P_\eps^\ell)^k f)
\le \left(\frac{3}{\bar\theta^k}
+\frac{30}{\bar\theta^{k}} k W(\ell-1)\rho^{1/p-1}\right)V_{1, 1/p}(f)+\\
+\left(30 W(\ell-1)\frac{1}{\bar\theta^{k}\delta_{k\ell}}+\frac{4\rho^{1-1/p}}{ \underline K_{k\ell} }\right)\|f\|_1,
\end{aligned}\end{equation}
where $\underline K_{k\ell}:=\sup_{0<\eps< \eps_0}\min_{J\in \mathcal P_\eps^{(k\ell)}}\{|T_{\eps}^{k\ell} J|\}$ and $\delta_{k\ell}>0$ by assumption i). 
Since $\ell$ is fixed we can choose $k$ large enough so that 
$$
\beta:=\left(\frac{3}{\bar\theta^k}
+\frac{30}{\bar\theta^{k}} k W(\ell-1)\rho^{1/p-1}\right)<1
$$
and let 
$$K:=\left(30 W(\ell-1)\frac{1}{\bar\theta^{k}\delta_{k\ell}}+\frac{4\rho^{1-1/p}}{ \underline K_{k\ell} }\right).$$
Thus, we have
\begin{equation}\label{ULYforTell}
V_{1, 1/p}((P_\eps^\ell)^k f)\le \beta V_{1, 1/p}(f)+K\| f\|_1.
\end{equation}
Similar to \eqref{step2}, by using \eqref{step1} and Lemma \ref{pvariation}, for any $j\in \mathbb N$ we have
 \begin{equation}\label{step3}
\begin{aligned}
V_{1, 1/p}(P_\eps^j f)\le \left(\frac{3}{(C\theta)^j}
+30 \frac{W(j-1)}{C\theta^{j-1}}\rho^{1/p-1}\right)V_{1, 1/p}(f)\\
+\left(30\frac{W(j-1)}{C\theta^{j-1}}\frac{1}{\delta_j}+\frac{4\rho^{1-1/p}}{ \underline K_j }\right)\|f\|_1.
\end{aligned}\end{equation}
Set $k_0=\ell k$, where $k$ and $\ell$ are chosen so that $P_\eps^{k_0}$ satisfies \eqref{step2}. Then for any $n\in\mathbb N$ 
we can write $n=k_0m+j$ for some   $j=1, ..., k_0-1.$ Applying \eqref{ULYforTell} consecutively implies
\begin{align*}
V_{1, 1/p}(P^n_\eps f) &=V_{1, 1/p}((P_\eps^{k_0})^m\circ  P_\eps^jf)\le 
\beta  V_{1, 1/p}((P_\eps^{k_0})^{m-1}\circ P_\eps^j f)+K\|f\|_1 \\ &\le ...\le
\beta^mV_{1, 1/p}(P_\eps^j f)+K\frac{\beta}{1-\beta}\|f\|_1.
\end{align*}
Using \eqref{step3} and setting  
$$
A_1:=\max_{1\le j\le k_0}\left\{\frac{3}{(C\theta)^j}
+30 \frac{W(j-1)}{C\theta^{j-1}}\rho^{1/p-1}\right\}\cdot\beta^{-j/k_0},
$$  
$$
A_2:=\max_{1\le j\le k_0}\left\{30\frac{W(j-1)}{C\theta^{j-1}}\frac{1}{\delta_j}+
\frac{4\rho^{1-1/p}}{ \underline K_j }\right\}+K\frac{\beta}{1-\beta} +1, \text{ and } \kappa:=\beta^{1/k_0},
$$
we obtain
$$
||P^n_\eps f||_{1,1/p} \le A_1\kappa^n ||f ||_{1, 1/p}+A_2\|f\|_1.
$$ 
\end{proof}
\subsubsection{Estimating the difference of the transfer operators in the mixed norm}
Define the following operator `mixed' norm: 
$$|||P_\eps|||:= \sup_{ \|f\|_{1, 1/p} \le 1}||P_\eps f||_1.$$
To apply the spectral stability result of  \cite{KL}, we still need to prove that $\lim_{\eps\to 0}|||P_\eps-P|||=0.$ Firstly, we start with a simple lemma that is similar\footnote{Lemma 11 in \cite{Kel82} was proved for $f\in BV$, the space of one dimensional functions of bounded variation. Here we deal with functions in $BV_{1, 1/p}$. To keep the paper self contained, we include a proof.} to Lemma 11 in \cite{Kel82}.  
\begin{lemma}\label{skeller} 
For any $f \in BV_{1, 1/p}(I)$ and $u\in L^1(I)$ 
$$
\left| \int_{I} f\cdot u dx \right| \le (1+\rho_0^{1/p}) \|f\|_{1, 1/p} \sup_{z\in I} \int_{x\le z} u(x)dx.
$$
\end{lemma}
\begin{proof}
We prove the lemma when $u$ is a simple function. Then general case follows, since $BV_{1, 1/p}\subset L^1$ and any $L^1$ function can be approximated by a sequence of simple functions. Let $-\frac{1}{2}=a_0 < a_2 < ... < a_n=\frac{1}{2}$ be a partition of $I$ and suppose that $u$ is constant on each $J_i=(a_{i-1}, a_{i}),$ $i=1, 2, ..., n.$ Let $\overline{\text{conv}(f(J_i))}$ denote the closure of the convex hull of $f(J_i)$ and set $G(x):=\int_{-\frac{1}{2}}^{x} udx$. Then 
\begin{equation*}\begin{aligned}
&\left| \int f\cdot udx \right| =\left| \sum_{i=1}^n \int_{J_i} f\cdot udx \right| = \left| \sum_{i=1}^n u|_{J_i} \int_{J_i} fdx \right| \\
&=\left| \sum_{i=1}^n u|_{J_i} y_i\right| \quad \text{where}  \quad y_i\in \overline{\text{conv}(f(J_i))}
\\
&=\left| \sum_{i=1}^n y_i \int_{J_i}  udx \right|=\left| \sum_{i=1}^n y_i(G(a_i)-G(a_{i-1}))\right| \\
&\le \sum_{i=2}^n |y_i-y_{i+1}||G(a_i)| +|G(-1/2)y_1|+|G(1/2)y_n|
\\
&\le \sup_{z\in I}|G(z)|\sum_{i=1}^n |y_i-y_{i+1}| +\|f\|_{\infty}\sup_{z\in I}|G(z)|.
\end{aligned}\end{equation*}
In the last inequality we used the facts $G(-1/2)=0$ and $G(1/2)\le \sup_{z\in I}|G(z)|.$ 
By definition of $\osc(f, \rho, x)$ we have
$$
|y_i-y_{i-1}|\le  \osc(f|_{J_i}, 1, z)
$$
which implies 
$$
\int \sum_{i=1}^n |y_i-y_{i+1}|dx\le \osc_1(f, 1) \le \rho_0^{1/p}V_{1, 1/p}(f).
$$
Substituting the latter into above equation finishes the proof.
\end{proof}
\begin{lemma}\label{triplemorm0} 
$$
\lim_{\eps\to 0}|||P_\eps-P_0|||=0.
$$
\end{lemma}
\begin{proof} 
%
For any $f \in {BV}_{1, 1/p}$ we have 
\begin{equation}\label{triplenorm}\begin{aligned}
&\|P_\eps f -Pf \|_1 \le  \|P_\eps f -P_{\eps}(\textbf{1}_{H^c}f) \|_1 +\|P_{\eps}(\textbf{1}_{H^c} f) -P(\textbf{1}_{H^c}f) \|_1\\
&\hskip 6cm+\|P(\textbf{1}_{H^c} f) -Pf \|_1.
\end{aligned}\end{equation}
We first estimate the first and the last term in \eqref{triplenorm}. By linearity of $P$ we have 
\begin{equation}\begin{aligned}\label{PTh-PT}
\|P(\textbf{1}_{H^c} f )-Pf \|_1&= \| P(\textbf{1}_{H^c}f- f)\|_1\le\| \textbf{1}_{H}f\|_1\\
&\le \|f\|_{1, 1/p}|H|\le 2|O_{\eps_0}| \|f\|_{1, 1/p} \le  \eta\|f\|_{1, 1/p}.
\end{aligned}
\end{equation}
Similarly, for the first term we have 
\begin{equation}\label{PTh-PTeps}
 \|P_\eps f -P_{\eps}(\textbf{1}_{H^c}f )\|_1\le \eta\|f\|_{1, 1/p}.
\end{equation}
It remains to prove that the second term in  equation \eqref{triplenorm} goes to zero as $\eps\to 0.$  Let $u=\text{sgn}(P_T({\bf 1}_{H^c}f)-P_{T_\eps}({\bf 1}_{H^c}f))$. Using the dual operators of $P$, $P_\eps$ and Lemma \ref{skeller}, we have: 
\begin{equation}\begin{aligned}\label{PT-PTeps}
\|P(\textbf{1}_{H^c} f) -P_{\eps}(\textbf{1}_{H^c}f) \|_1 =
\left |\int u(P(\textbf{1}_{H^c} f) -P_{\eps}(\textbf{1}_{H^c}f))dx\right | \\
=\left |\int \textbf{1}_{H^c} u\circ T f dx- \int \textbf{1}_{H^c}u\circ T_\eps  f dx\right | =
\left |\int \textbf{1}_{H^c} f (u\circ T-u\circ T_\eps)dx\right | \\
\le (1+\rho_0^{1/p})\|f\|_{1, 1/p}\sup_{z}\int_{-1/2}^z{\bf 1}_{H^c}(u\circ T-u\circ T_\eps)dx \\
= 
(1+\rho_0^{1/p})\|f\|_{1, 1/p}\sup_{z}\int_{[-1/2, z]\cap H^c}(u\circ T-u\circ T_\eps)dx \\
=(1+\rho_0^{1/p})\|f\|_{1, 1/p}\sup_{z}\sum_{i=1}^2\left( \int _{T_i(H_z^c)}\frac{u(y)}{T'(T^{-1}_i(y))}dy-
\int_{T_{i, \eps}(H_z^c)}\frac{u(y)}{T_\eps'(T^{-1}_{i, \eps}(y))} dy\right)
\end{aligned}\end{equation}
where $H_z^c=[-1/2, z]\cap H^c$ and we used change of variables $y=T(x)$ and $y=T_\eps(x)$ for the first and second summands respectively. For $i=1, 2$ we have 
\begin{equation}\begin{aligned}\label{intt-intteps}
&\int _{T_i(H_z^c)}\frac{u(y)}{T'(T^{-1}_i(y))}dy-
\int_{T_{i, \eps}(H_z^c)}\frac{u(y)}{T_\eps'(T^{-1}_{i, \eps}(y))} dy \\
&= 
\int _{T_i(H_z^c)\cap T_{i, \eps}(H_z^c)}
\left(\frac{u(y)}{T'(T^{-1}_i(y))} - \frac{u(y)}{T_\eps'(T^{-1}_{i, \eps}(y))} \right)dy \\
&+\int _{T_i(H_z^c)\setminus T_{i, \eps}(H_z^c)}\frac{u(y)}{T'(T^{-1}_i(y))}dy-
\int _{T_{i, \eps}(H_z^c)\setminus T_{i}(H_z^c)} \frac{u(y)}{T_\eps'(T^{-1}_{i, \eps}(y))} dy\\
&:= E_1+E_2+E_3
\end{aligned}\end{equation}
Now we estimate each of the terms in  the right hand side separately. Using  $T'_\eps\ge  C\theta$ and the fact that $||u||_{\infty}\le 1$ we have 
\begin{equation*}\begin{aligned}
E_1 &:= \left| \int _{T_i(H_z^c)\cap T_{i, \eps}(H_z^c)}
\left(\frac{u(y)}{T'(T^{-1}_i(y))}  -  \frac{u(y)}{T_\eps'(T^{-1}_{i, \eps}(y))} \right)dy \right|\\
&\le \int_{T_i(H_z^c)\cap T_{i, \eps}(H_z^c)}\left|\frac{1}{T'(T^{-1}_i(y))}  -  \frac{1}{T'(T^{-1}_{i, \eps}(y))} \right|dy \\
&+ \int_{T_i(H_z^c)\cap T_{i, \eps}(H_z^c)}\left|\frac{1}{T'(T^{-1}_{i, \eps}(y))}  -  \frac{1}{T_\eps'(T^{-1}_{i, \eps}(y))} \right|dy\\
&\le\|\frac{1}{T'}\|_\alpha \int _{T_i(H_z^c)\cap T_{i, \eps}(H_z^c)}|T^{-1}_{i, \eps}(y)- T^{-1}_i(y)|^\alpha dy \\
&+\frac{1}{(C\theta)^2}\int _{T_i(H_z^c)\cap T_{i, \eps}(H_z^c)}|T_\eps'(T^{-1}_{i, \eps}(y))- T'(T^{-1}_{i, \eps}(y))|dy \\
&\le \|\frac1{T'}\|_{\alpha}\int_{T_i(H_z^c)\cap T_{i, \eps}(H_z^c)}|T^{-1}_{i, \eps}(y)- T^{-1}_i(y)|^\alpha dy\\
&+ \frac{1}{(C\theta)^2}|T_i(H_z^c)\cap T_{i, \eps}(H_z^c)|d(T, T_\eps).
\end{aligned}\end{equation*}
Notice that for $y\in T(H^c_z)\cap T_\eps(H_z^c)$  there exists $x\in  H_z^c$ such that $T_\eps(x)=y$ 
we have 
\begin{equation}\label{t-teps}
|T^{-1}_{i, \eps}(y)- T^{-1}_i(y)|= |x-T_i^{-1}(T_\eps(x))|\le\Lip (T^{-1}_i)d(T, T_\eps) \le \frac{1}{ C\theta}d(T, T_\eps).
\end{equation}
Taking into account the fact that $|T_i(H_z^c)\cap T_{i, \eps}(H_z^c)|\le 1$  the relation \eqref{t-teps} implies that 
$$
E_1 \le C_2\eta^\alpha.
$$
Now note that by assumption f) for any $x\in H^c$ we have $|T_\eps(x)-T(x)|< C\eta$. This implies 
$|T_\eps(H^c) \bigtriangleup T(H^c)| \le 2C\eta$. 
Hence, for all sufficiently small $\eps$ we have  
$$
E_2:=\left| \int _{T_i(H_z^c)\setminus T_{i, \eps}(H_z^c)}\frac{u(y)}{T'(T^{-1}_i(y))}dy \right|
 \le  \frac{1}{C\theta} |T_i(H_z^c)\setminus T_{i, \eps}(H_z^c)|  =C_3\eta.
$$ 
Similarly, 
$$
E_3:= \left| \int _{T_{i, \eps}(H_z^c)\setminus T_{i}(H_z^c)} \frac{u(y)}{T_\eps'(T^{-1}_{i, \eps}(y))} dy \right| \le C_3\eta.
$$
Substituting estimates for $E_1$, $E_2$ and $E_3$ first into equation \eqref{intt-intteps}  and  then substituting the result into \eqref{PT-PTeps} gives  
$$
\|P(\textbf{1}_{H^c} f) -P_{\eps}(\textbf{1}_{H^c}f) \|_1 \le  (1+\rho_0^{1/p})(2C_2+4C_3)\|f\|_{1, 1/p}\eta^{\alpha}.
$$
Substituting this and equations \eqref{PTh-PT} and \eqref{PTh-PTeps} into  \eqref{triplenorm}  implies
$$
\|P_\eps f -Pf \|_1\le C\eta^\alpha\|f\|_{1, 1/p}
$$
which finishes the proof. 
 \end{proof}
 We are now ready to prove that the $1$-d family $T_\eps$ is strongly statistically stable. Firstly, we set some notation. Consider the set
$$V_{\delta,r}(P)=\{z\in\mathbb C: |z|\le r\text{ or dist}(z,\sigma(P))\le\delta\},$$
where $\sigma(P)$ is the spectrum of $P$ when acting on $BV_{1,1/p}$.
 \begin{proposition}\label{1dstability}
$$ \lim_{\eps\to 0}||h_{\eps}-h||_1=0.$$
 \end{proposition}
 \begin{proof}
 By Lemma \ref{uniform} and Lemma \ref{triplemorm0}, for any $z\in V_{\delta,r}(P)$ the Keller-Liverani \cite{KL} stability result implies
 $$\lim_{\eps \to 0}|||(z I-P_{\eps})^{-1}-(z I-P)^{-1}|||=0.$$
 Consequently,
 $$\lim_{\eps \to 0}|||\Pi_{1,\eps}- \Pi_{1}|||=0,$$
 where $\Pi_{1,\eps}$ and $\Pi_{1}$ are the spectral projections of $P_\eps$ and $P$ associated with the eigenvalue $1$. This completes the proof since both $\Pi_{1,\eps}$ and $\Pi_{1}$ have rank $1$. 
 \end{proof}
\subsubsection{Statistical stability: from the 1-d family, to the Poincar\'e maps, to the family of flows}\label{ssfrom1d}
We now discuss how to obtain continuity of the SRB measures (3) of Theorem \ref{main}) from Proposition \ref{1dstability}. 
We first show how the absolutely continuous invariant measures of the family of $1$-d maps are related to the SRB measures of the family of the flows via the Poincar\'e maps. This construction is well known (see for instance \cite{APPV}). Let $\tpsi: \Sigma\to \mathbb R$ be any bounded function. Notice that $\Sigma$ is foliated by stable manifolds, and any $x\in I$ defines unique stable manifold $\pi^{-1}(x)$. Therefore $\psi^{+}_\eps:I\to \mathbb R$ and $\psi^{-}_\eps:I\to \mathbb R$ are well defined by 
$$
\psi^{+}_{\eps}(x):=\sup_{\xi\in\pi_{\eps}^{-1}(x)}\tpsi(\xi) \quad\text{and} \quad \psi_{\eps}^{-}(x):=\inf_{\xi\in\pi_{\eps}^{-1}(x)}\tpsi(\xi).
$$
There exists a unique  $F_\eps$-invariant probability measure $\mu_{F_\eps}$ on $\Sigma$ such that for every continuous function $\tpsi:\Sigma \to \mathbb R$
	$$
	\int \tpsi d\mu_{F_\eps}=\lim_{n\to \infty}\int_{I}(\tpsi\circ F_\eps^n)^{+}_\eps d\bar\mu_\eps=\lim_{n\to \infty}\int_{I}(\tpsi\circ F_\eps^n)^{-}_\eps d\bar\mu_\eps,
	$$
where $\bar\mu_\eps$ is the $T_\eps$-invariant absolutely continuous measure (see for instance, Lemma 6.1, \cite{APPV}). To pass from the Poincar\'e map to the flow we use standard  procedure: first consider suspension flow from the Poincar\'e map and  then embed the suspension flow into the original flow. To apply the  construction we first need to prove the following 
\begin{lemma}
 For every  $X_\eps\in\mathcal X$  let $F_\eps:\Sigma\to \Sigma$ be its Poincar\'e map and define $\mu_{F_\eps}$ as above. Then $\tau_\eps$ is $\mu_{F_\eps}$- integrable.
\end{lemma}
\begin{proof}
Let $\tau_{N, \eps}=\min\{N, \tau_\eps\}$. Then $\tau_{N, \eps}$ is monotone increasing in $N$ and it converges to $\tau_\eps$ almost everywhere. 
Since $\tau_{N, \eps}$ is continuous and $d\bar\mu_\eps/dm$ is uniformly bounded, in $\eps$, we have 
\begin{align*}
\int \tau_{N, \eps} d\mu_{F_\eps}=\lim_{n\to \infty}\int (\tau_{N, \eps}\circ F_\eps^n)_\eps^{+}d\bar\mu_\eps =\lim_{n\to \infty}\int (\tau_{N, \eps}\circ F^n)_\eps^{+}(d\bar\mu_\eps/ dm)dm\\
\le C\|d\bar\mu_\eps/dm\|_\infty|\int_{I} \log |x- O_\eps|dx|<+\infty, 
\end{align*}
which  implies that $\lim_{N\to \infty}\int \tau_{N, \eps} d\mu_F$ exists and finite. Hence by monotone convergence  theorem $\int\tau_{\eps}d\mu_{F_\eps}<\infty.$
\end{proof}
\noindent Let 
$$
\Sigma_{\tau_\eps}=\Sigma\times [0, +\infty)/\sim,
$$
where  $\sim$  is the equivalence relation on $\Sigma\times [0, +\infty)$ generated by $(\xi, \tau_\eps(\xi))\sim (F_\eps(\xi), 0).$
Then there is a natural projection $\pi_{\tau_\eps} : \Sigma\times [0, \infty)\to \Sigma _{\tau_\eps}$ which induces a topology and a Borel $\sigma$-algebra on $\Sigma_{\tau_\eps}.$ The suspension flow of $F_\eps$ with return time $\tau_\eps$ is the semi-flow $(X_\eps^t)_{t\ge 0}$ defined on $\Sigma_{\tau_\eps}$ as 
$$
X_\eps^t(\pi_{\tau_\eps}(\xi, s))=\pi_{\tau_\eps}(\xi, s+t) \quad \text{for any} \quad (\xi, s)\in \Sigma\times [0, +\infty).
$$	
\begin{lemma} (\cite{APPV}, Lemma 6.7.)
 The suspension flow $X_\eps^t$ admits a unique invariant probability measure $\bar\mu_{X_\eps}$. Moreover, for every bounded measurable $\varphi: \Sigma_{\tau_\eps}\to \mathbb R,$ we have
$$
\int\varphi d\bar \mu_{X_\eps}=\frac {1}{\mu_{F_\eps}(\tau_\eps)}\int\int_0^{\tau_\eps(\xi)}\varphi(\pi_{\tau_\eps}(\xi, t))dtd\mu_{F_\eps}(\xi),
$$ 
where $\mu_{F_\eps}(\tau)=\int\tau_\eps d\mu_{F_\eps}.$
\end{lemma}
Now, we can define the unique SRB measure of  the original flow $X_\eps(\xi, t)$. Define 
$$
\Phi_\eps :\Sigma\times [0, +\infty)\to  U \quad \text{by letting} \quad \Phi_\eps(\xi, t)=X_\eps(\xi, t).
$$  	
Since $\Phi_\eps(\xi, \tau_\eps(\xi))=(F_\eps(\xi), 0)\in \Sigma\times\{0\}$, map $\Phi_\eps$ induces a map
\begin{equation}
\phi_\eps:\Sigma_{\tau\eps}\to U, \quad \text{such that} \quad  \phi_\eps\circ X_\eps^t=X_\eps(\cdot, t)\circ \phi_\eps, \quad \text{for} \quad t\ge 0.
\end{equation}
via the identification $\sim$. Now the invariant measure of $X_\eps^t$  is naturally transferred to an invariant measure for $X_\eps(\cdot, t)$ via pushing it forward  $\mu_{\eps}={\phi_{\eps}}_\ast\bar\mu_{X_\eps}$ (see \cite{APPV},  Section 7). Hence, we can define the SRB measure of the flow as follows:
\begin{lemma}\label{measureforflow}
The flow of each $X_\eps\in\mathcal X$ has a unique SRB measure $\mu_{\eps}$. In particular, for any continuous function $\varphi:U\to \mathbb R$
$$
\int\varphi d\mu_{\eps}=\frac {1}{\mu_{F_\eps}(\tau_\eps)}\int\int_0^{\tau_\eps(\xi)}\varphi \circ \phi_\eps \circ \pi_\eps(\xi, t)dtd\mu_{F_\eps}(\xi)
$$ 
where $\mu_{F_\eps}(\tau_\eps)=\int\tau_\eps d\mu_{F_\eps}.$
\end{lemma}
The proof of  2) of Theorem \ref{main}, then proceeds as follows. We first note that Lemma \ref{uniform} implies that the densities $h_\eps$ are in $BV_{1,1/p}$ and hence in $L^{\infty}$. Then by our Proposition \ref{1dstability} above and Proposition 3.3 of \cite{AS} we obtain that the Poincar\'e map is statistically stable.   Then statistical stability of the Poincar\'e map is first lifted to the suspension flow and finally to the original flow.  Notice that the key ingredients in the proof of Propositions 3.3  and  Lemma 4.2. in \cite{AS} are that the densities $h_\eps$ are in $L^{\infty}$, the compactness of the Poincar\'e section and the fact that the Lebesgue measure of the set where $\tau_\eps> n$ decays sufficiently fast,  which is the case of our setting.  

\end{document}